\documentclass[12pt,a4paper]{article}
\usepackage{amsmath,amssymb,centernot, amsthm,url}

\newtheorem{thm}{Theorem}[section]
\newtheorem{lem}[thm]{Lemma}

\newtheorem{cor}[thm]{Corollary}
\newtheorem{conj}[thm]{Conjecture}

\newtheorem{defi}[thm]{Definition}

\newtheorem{re}[thm]{Result}
\newtheorem{question}[thm]{Question}

\newtheorem{example}[thm]{Example}

\newtheorem{remark}[thm]{Remark}
\newenvironment{rmk}{\begin{remark} \em}{\end{remark}}

\DeclareMathOperator{\BH}{ {\rm BH} }
\DeclareMathOperator{\ord}{ {\rm ord} }

\DeclareMathOperator{\lcm}{ {\rm lcm} }

\begin{document}
\title{Necessary Conditions for the Existence of Group-Invariant Butson Hadamard Matrices and
a New Family of Perfect Arrays}
\author{Tai Do Duc\\ Division of Mathematical Sciences\\
School of Physical \& Mathematical Sciences\\
Nanyang Technological University\\
Singapore 637371\\
Republic of Singapore}

\maketitle

\begin{abstract}
Let $G$ be a finite abelian group and let $\exp(G)$ denote the least common multiple of the orders of all elements of $G$. A $\BH(G,h)$ matrix is a $G$-invariant $|G|\times |G|$ matrix $H$ whose entries are complex $h$th roots of unity such that $HH^*=|G|I_{|G|}$. By $\nu_p(x)$ we denote the $p$-adic valuation of the integer $x$. Using bilinear forms over abelian groups,  we \cite{ducschmidt} constructed
new classes of $\BH(G,h)$ matrices under the following conditions.
\begin{itemize}
\item[(i)] $\nu_p(h) \geq \lceil \nu_p(\exp(G))/2 \rceil $ for any prime divisor $p$ of $|G|$, and
\item[(ii)] $\nu_2(h) \geq 2$ if $\nu_2(|G|)$ is odd and $G$ has a direct factor $\mathbb{Z}_2$.
\end{itemize}
The purpose of this paper is to further study
the conditions on $G$ and $h$ so that a $\BH(G,h)$ matrix exists.
We will focus on $\BH(\mathbb{Z}_n,h)$ and $\BH(G,2p^b)$ matrices,
where $p$ is an odd prime.
Combining our work with previously known results, there are $2687$ open cases left for the existence of $\BH(\mathbb{Z}_n,h)$
matrices in which $1\leq n,h \leq 100$.
Finally, we show that $\BH(\mathbb{Z}_n,h)$ matrices can be used to
construct a new family of perfect polyphase arrays.
\end{abstract}

\section{Introduction}
Let $n$ and $h$ be positive integers. An $n\times n$ matrix $H$ whose entries are complex $h$th roots of unity is called a \textbf{Butson Hadamard matrix} if $HH^* = nI$,
where $H^*$ is the complex conjugate transpose of $H$ and $I$ is the identity matrix of order $n$.
We also say that $H$ is a {\boldmath$\BH(n,h)$} matrix.

Let $(G,+)$ be a finite abelian group of order $n$. An $n\times n$ matrix $A=(a_{g,k})_{g,k\in G}$ is {\boldmath $G$}\textbf{-invariant} if $a_{g+l,k+l}=a_{g,k}$ for all $g,k,l\in G$. A $G$-invariant $\BH(n,h)$ matrix is also called a {\boldmath$\BH(G,h)$} matrix. Note that in the case $G=\mathbb{Z}_n$, a cyclic group of order $n$, a $\BH(\mathbb{Z}_n,h)$ matrix is a \textbf{circulant} matrix, i.e., a matrix each of whose rows (except the first) is obtained from the previous row by shifting one position to the right and moving the last entry to the front.

\medskip

For any multiple $h'$ of $h$, a $\BH(G,h)$ matrix is also a $\BH(G,h')$ matrix, as each $h$th root of unity is automatically a $h'$th root of unity. Therefore, it is important to find those positive integers $h$ such that a $\BH(G,h)$ matrix exists, but $\BH(G,k)$ matrices do not exist for any divisor $k$ of $h$. Group-invariant Butson Hadamard matrices link to many other combinatorial objects like generalized Hadamard matrices, relative difference sets, generalized Bent functions, cyclic $n$-roots, see \cite{sc1}, and perfect polyphase arrays.

\medskip

A sequence $\{a_0,\dots,a_{n-1}\}$ is called a \textbf{perfect {\boldmath$h$}-phase sequence} of length $n$ if each $a_i$ is a complex $h$th root of unity and
$$\sum_{i=0}^{n-1}a_i\overline{a_{i+j}}=0 \ \text{whenever} \ j\neq 0,$$ 
where the indices are taken modulo $n$. Such a sequence is equivalent to a $\BH(\mathbb{Z}_n,h)$ matrix whose first row is $(a_0,\dots,a_{n-1})$. More generally, a multi-dimensional array $A=(a_{i_1,\dots,i_k})$ of size $n_1\times\dots\times n_k$ is called a \textbf{perfect \boldmath$h$-phase array} if its entries are complex $h$th roots of unity and
$$\sum_{0\leq i_j\leq n_j-1 \ \forall \ j} a_{i_1,\dots,i_k}\overline{a}_{i_1+s_1,\dots,i_k+s_k}=0$$
 whenever $(s_1,\dots,s_k) \neq (0,\dots,0)$, where the indices are taken modulo $n_j$ for $1\leq j \leq k$. The values 
$$R_{s_1,\dots,s_k}=\sum_{i_1,\dots,i_k}a_{i_1,\dots,i_k}\bar{a}_{i_1+s_1,\dots,i_k+s_k}$$
are called \textbf{autocorrelations} of the array $A$. In Lemma \ref{equivalence}, we show that such an array is equivalent to a $\BH(\mathbb{Z}_{n_1}\times\dots\times \mathbb{Z}_{n_k},h)$ matrix. Perfect sequences and perfect arrays have a wide range of applications in communication and radar systems, see \cite{bomer}, \cite{fan}, \cite{golomb} for example. In the last section of this paper, we use (circulant) $\BH(\mathbb{Z}_n,h)$ matrices to construct a new family of perfect polyphase arrays. Note that by \cite[Corollary $2.5$]{ducschmidt}, there exist $\BH(\mathbb{Z}_n,h)$ matrices whenever
\begin{itemize}
\item[$(i)$]  $\nu_p(h) \geq \lceil \nu_p(n)/2 \rceil$ for every prime divisor $p$ of $n$, and
\item[$(ii)$] $\nu_2(h) \ge 2$  if  $n\equiv 2\ (\bmod\ 4)$,
\end{itemize}
where $\nu_p(x)$ denotes the $p$-adic valuation of the integer $x$.

\medskip

\noindent In this paper, we focus on studying necessary conditions for the existence of the following two types of Butson Hadamard matrices:
\begin{itemize}
\item $\BH(\mathbb{Z}_n,h)$ matrices and
\item $\BH(G,2p^b)$ matrices, where $p$ is an odd prime.
\end{itemize}
The main tools used in this paper are group-ring equations combined with techniques from the field-descent method \cite{sc2, sc3}, and  upper bounds on the norm of cyclotomic integers. 
We sketch the approach as follows.

\medskip

Let $(a_g)_{g\in G}$ be the first row of a $\BH(G,h)$ matrix $H$. Let $D=\sum_{g\in G} a_gg$ be an element of the group ring $\mathbb{Z}[\zeta_h][G]$. By Result \ref{equiv} of Section 2 below, the equation $HH^*=|G|I$ is equivalent to
\begin{equation} \label{group ring}
DD^{(-1)}=|G|.
\end{equation}
Put $n=|G|$ and $m=\lcm(\exp(G),h)$. Let $\chi$ be a character of $G$ and put $X=\chi(D) \in \mathbb{Z}[\zeta_m]$. The equation (\ref{group ring}) implies
\begin{equation} \label{resulting weil}
|X|^2=n, \ X \in \mathbb{Z}[\zeta_m].
\end{equation}

By the field-descent method, there exists an integer $j$ such that $X\zeta_m^j$ belongs to a proper subfield $\mathbb{Q}(\zeta_k)$ of $\mathbb{Q}(\zeta_m)$. Note that $X\zeta_m^j$ is a sum of roots of unity in $\mathbb{Q}(\zeta_m)$. As $X\zeta_m^j$ belongs to the subfield $\mathbb{Q}(\zeta_k)$, usually a lot of the roots of unity in the sum cancel out. The norm of the remaining terms can be bounded using our result on the norm of cyclotomic integers. 
On the other hand,  by (\ref{resulting weil}), this norm is equal to $\sqrt{n}$. 
We thus obtain a restriction on $n$ and $h$.

\medskip

\noindent To end this section, we give a summary of known results on the existence of Butson Hadamard matrices.

\begin{re} \label{nonexistence} The following Butson Hadamard matrices do not exist.
\begin{itemize}
\item[1.] \cite[Leung-Schmidt]{leu1} $\BH(\mathbb{Z}_{2p^2},2p)$ matrices with $p$ being an odd prime.
\item[2.] \cite[Ma-Ng]{mang} $\BH(\mathbb{Z}_{3pq},3)$ matrices with $p, q >3$ being distinct primes.
\item[3.] \cite[Hiranandani-Schlenker]{hirasch} $\BH(\mathbb{Z}_{p+q},pq)$ matrices with $p,q >3$ being distinct primes.
\end{itemize}
\end{re}

\begin{re} \label{sylvester conditions}
\cite[Sylvester conditions]{hirasch} 
\begin{itemize}
\item[(i)] If a $\BH(n,2)$ matrix exists, then $n=2$ or $4 \mid n$. Furthermore, if a $\BH(\mathbb{Z}_n,2)$ matrix exists, then $4\mid n$ and $n/4$ is a square.
\item[(ii)] Let $p\geq 3$ be a prime. If a $\BH(p+2,h)$ matrix exists, then $h$ does not have the form $2p^b$ for some positive integer $b$.
\item[(iii)] Let $q\geq 3$ be a prime. If a $\BH(2q,h)$ matrix exists, then $h$ does not have the form $2^ap^b$ for $a,b\in \mathbb{Z}^+$ and any prime $p>q$.
\end{itemize}
\end{re}

\begin{re} \cite[Lam-Leung]{lam} \label{Lam-Leung}
If a $\BH(G,h)$ matrix exists and $p_1,\ldots,p_r$ are the distinct
prime divisors of $h$, then there exist nonnegative integers
$a_i$ such that $|G|=\sum_{i=1}^r a_ip_i$. 
\end{re}

 \begin{re} \cite[Brock]{brock} \label{Brock}
Let $n$ be an integer and let $m$ be the square-free part of $n$. Assume that $m$ is odd. If a $\BH(n,h)$ matrix exists, then $m$ has no prime factor $p$ which satisfies
\begin{itemize}
\item[(a)] $p$ does not divide $h$, and
\item[(b)] $p^j \equiv -1 \pmod{h}$ for some integer $j$.
\end{itemize}
\end{re}

\begin{re}\cite[Duc-Schmidt]{ducschmidt} \label{TaiSchmidt}
Let $G$ be an abelian group and $h$ be a positive integer. Then a $\BH(G,h)$ matrix exists if
\begin{itemize}
\item[(i)] $v_p(h) \geq \lceil v_p(\exp(G)/2) \rceil$ for every prime divisor $p$ of $|G|$, and
\item[(ii)] $v_2(h) \geq 2$ if $v_2(|G|)$ is odd and $G$ has a direct factor $\mathbb{Z}_2$.
\end{itemize}
Moreover, if $G=\mathbb{Z}_{p^a}$ is a cyclic group of prime-power order, 
then (i) and (ii) are also necessary conditions for the existence of $\BH(\mathbb{Z}_{p^a},h)$ matrices.
\end{re}

 \noindent We remark that the existence problem for $\BH(\mathbb{Z}_n,h)$ matrices is in general a difficult problem. For example, the existence of $\BH(\mathbb{Z}_n,2)$ matrices is equivalent to the existence of circulant Hadamard matrices. Ryser \cite[p. 134]{rys} conjectured that circulant Hadamard matrices exist only for orders $n=1$ or $n=4$. This conjecture remains open for more than $50$ years.

\bigskip


\section{Preliminaries}
\subsection{Notations and Group Rings}
In this section, we fix some notation, state definitions and known results which will be used later. 
We start with some notations.
\begin{itemize}
\item[1.] For a positive integer $h$, we use $\zeta_h$ to denote a primitive $h$th root of unity.
\item[2.] For a prime $p$ and an integer $n$, let $\nu_p(n)$ denote the $p$-adic valuation of $n$, that is, the
largest nonnegative integer $x$ such that $p^x$ divides $n$.
\item[3.] A positive integer is \textbf{square-free} if it is not divisible by any square of a prime. We call an integer $l$ the \textbf{square-free part} of $n$ if $l$ is the product of all prime divisors $p$ of $n$ in which $\nu_p(n)$ is odd.
\item[4.] For coprime integers $m$ and $n$, we denote the smallest positive integer $j$ such that $m^j\equiv 1 \pmod{n}$ by $\ord_n(m)$.
\item[5.] We call a complex number $X$ a \textbf{cyclotomic integer} if it is a sum of complex roots of unity.
\end{itemize}

\noindent As it turns out, group-ring equations are pivotal in our study. 

Let $G$ be a finite abelian group of order $n$. Let $R$ be a ring with identity $1$ and let $R[G]$ denote the group ring of $G$ over $R$. Any element $X \in R[G]$ is uniquely expressed as $X=\sum_{g \in G} a_gg$, $a_g \in R$. We call the set $\{g\in G: a_g\neq 0\}$ the \textbf{support} of $X$ and denote it by $\text{supp}(X)$. Two elements $X=\sum_{g \in G} a_gg$ and $Y=\sum_{g\in G}b_gg$ are equal if and only if $a_g=b_g$ for all $g\in G$. We denote $1_G$ as the identity element of $G$. In the case where there is no confusion, we simply write $c$ in the place of $c1_G$. A subgroup $U$ of $G$ is identified with $\sum_{u \in U} u$ in $R[G]$. 

The group of complex characters of $G$ is denoted by $\hat{G}$. It is well known that $\hat{G}$ is isomorphic to $G$. The trivial character of $G$, denoted by $\chi_0$, is defined by $\chi_0(g)=1$ for all $g\in G$. For $D=\sum_{g\in G} a_gg \in R[G]$ and $\chi \in \hat{G}$, write $\chi(D)=\sum_{g\in G} a_g\chi(g)$. Let $U$ be a subgroup of $G$, denote
$$U^\perp=\{\chi \in \hat{G}: \chi(g)=1 \ \forall \ g \in U\}.$$
Note that $U^\perp$ is isomorphic to $\widehat{G/U}$. The following result is from \cite[Lemma 2.8]{feng}.

\begin{re} \label{orthogonal sum}
Let $G$ be a finite abelian group, let $U$ be a subgroup of $G$, and let
$D=\sum_{g \in G} a_gg \in \mathbb{C}[G]$. Then, for any character $\chi \in \hat{G}$, we have
$$\sum_{\tau \in U^\perp}\chi\tau(D)=|U^\perp|\chi\left(\sum_{g\in U}a_gg\right)$$
\end{re}

The next result is called Fourier inversion formula. 
A proof  can be found in \cite[Chapter VI, Lemma 3.5]{bethjung}, for example.

\begin{re} \label{fourier inversion}
Let $G$ be a finite abelian group and let $\hat{G}$ denote the group of characters of $G$. Let $D=\sum_{g\in G} a_gg \in \mathbb{C}[G]$. Then
$$a_g=\frac{1}{|G|}\sum_{\chi \in \hat{G}} \chi(Dg^{-1}) \ \forall \ g\in G.$$
Consequently, if $D, E \in \mathbb{C}[G]$ and $\chi(D)=\chi(E)$ for all $\chi \in \hat{G}$, then $D=E$.
\end{re}

\medskip

In our study, we focus on the ring $R=\mathbb{Z}[\zeta_h]$. Let $D=\sum_{g \in G}a_gg, a_g \in R$ for all $g\in G$, be an element of $R[G]$. Let $t$ be an integer coprime to $h$ and let $\sigma \in \text{Gal}(\mathbb{Q}(\zeta_h)/\mathbb{Q})$ be defined by $\zeta_h^\sigma=\zeta_h^t$. Let $D^{(t)}$ denote
$$D^{(t)}=\sum_{g\in G} a_g^\sigma g^t.$$

\medskip

As mentioned in the introduction, 
a $\BH(G,h)$ matrix is equivalent to a solution of the
group ring equation $DD^{(-1)}=|G|$ 
for some $D\in \mathbb{Z}[\zeta_h][G]$. 
We now state this result formally, see \cite[Lemma 3.3]{ducschmidt} for a proof.

\begin{re} \label{equiv}
Let $G$ be a finite abelian group, let $h$ be a positive integer, 
and let $a_g, g\in G$, be integers. Consider the element $D=\sum_{g\in G}\zeta_h^{a_g}g$ of $\mathbb{Z}[\zeta_h][G]$ and the $G$-invariant matrix $H=(H_{g,k}), g,k\in G$, given by $H_{g,k}=\zeta_h^{a_{k-g}}$. Then $H$ is a $\BH(G,h)$ matrix if and only if
$$DD^{(-1)}=|G|.$$
\end{re}

\medskip

The next result is a generalization of Ma's lemma, see \cite{ma} or \cite[Lemma 1.5.1]{sc2}. The proof of this result is similar to the proof of the original result. We provide it here for the convenience of the reader.

\begin{re} \label{ma lemma}
Let $p$ be a prime and let $G$ be a finite abelian group whose Sylow $p$-subgroup $S$ is cyclic. Let $P$ be the subgroup of $S$ of order $p$. Let $t, h$ be positive integers such that $h$ is not divisible by $p$. If $D\in \mathbb{Z}[\zeta_h][G]$ satisfies
\begin{equation} \label{congruence assumption}
\chi(D) \equiv 0 \pmod{p^t}
\end{equation}
for all characters $\chi \in \hat{G}$ of order divisible by $|S|$, then there exist $X,Y \in \mathbb{Z}[\zeta_h][G]$ such that
\begin{equation} \label{form of D}
D=p^tX+PY.
\end{equation}
\end{re}

\begin{proof}
Assume $|S|=p^s$ for some $s\in \mathbb{Z}^+$. We first prove the result for $D\in \mathbb{Z}[\zeta_h][S]$.
Let $g$ be a generator of the cyclic group $S$.
Define a ring homomorphism
\begin{eqnarray*}
   \rho: \mathbb{Z}[\zeta_h][S] & \rightarrow & \mathbb{Z}[\zeta_{hp^s}], \\
  \sum_{i=0}^{p^s-1}a_i g^i &\mapsto& \sum_{i=0}^{p^s-1}a_i\zeta_{p^s}^i
\end{eqnarray*}
for all $a_0,\ldots,a_{p^s-1}\in \mathbb{Z}[\zeta_h]$. 
We claim that
\begin{equation} \label{kernel of rho}
\ker(\rho)=\{PY: Y \in \mathbb{Z}[\zeta_h][S]\}.
\end{equation}

Let $\sum_{i=0}^{p^s-1}a_ig^i \in \ker(\rho)$, then $\sum_{i=0}^{p^s-1}a_i\zeta_{p^s}^i=0$. Since $p$ does not divide $h$, the minimal polynomial of $\zeta_{p^s}$ over $\mathbb{Z}[\zeta_h]$ is $\phi(x)=1+x^{p^{s-1}}+\cdots+x^{p^{s-1}(p-1)}$. We obtain
$$\sum_{i=0}^{p^s-1}a_ix^i=\phi(x)f(x), \ f(x) \in \mathbb{Z}[\zeta_h][x],$$
so $\sum_{i=0}^{p^s-1}a_ig^i=\phi(g)f(g)=Pf(g)$, which proves (\ref{kernel of rho}).

Since $\rho(D) \equiv 0 \pmod{p^t}$, there exists $X\in \mathbb{Z}[\zeta_h][S]$ so that $\rho(D)=p^t\rho(X)$. We obtain $(D-p^tX) \in \ker(\rho)$ and hence
$$D=p^tX +PY \ \text{for some} \ Y \in \mathbb{Z}[\zeta_h][S].$$

\medskip

\noindent The result is proved for $D \in \mathbb{Z}[\zeta_h][S]$. Now, write $G=S\times H$ where $|H|\not\equiv 0\pmod{p}$ and write
$$D=\sum_{k \in H} D_kk, \ D_k\in \mathbb{Z}[\zeta_h][S].$$
Let $\psi$ be a character of $S$ of order $p^s$. 
For any character $\gamma$ of $H$, $\psi\gamma$ is a character of $G$ defined by $\psi\gamma(xy)=\psi(x)\gamma(y)$ for all $x\in S, \ y\in H$. Note that the order of $\psi\gamma$ is divisible by $p^s$. By (\ref{congruence assumption}), we have
   \begin{equation} \label{congruence equation}
     \psi\gamma(D)=\sum_{k \in H} \psi(D_k)\gamma(k) \equiv 0 \pmod{p^t}.
   \end{equation}
   
Define an $|H| \times |H|$ matrix by $M=(\gamma(k))_{\gamma \in \hat{H}, k\in H}$ and 
 a column vector $u$ by $u=(\psi(D_k))_{k\in H}^T$. As the congruence (\ref{congruence equation}) holds for any character $\gamma$ of $H$, all entries of $N:=Mu$ are divisible by $p^t$. 
By the orthogonality relations for characters, 
we have $M^{-1}=M^{*}/|H|$ where $M^*$ is the complex conjugate transpose of $M$. 
Thus $u = M^{-1}N = M^{*}N/|H|$. As $\gcd(p,|H|)=1$ and all entries of $N$
are divisible by $p^t$, we obtain
\begin{equation} \label{congruence equation for components}
\psi(D_k) \equiv 0 \pmod{p^t} \ \text{for all} \ k \in H.
\end{equation}

The equation (\ref{congruence equation for components}) holds for every character $\psi$ of $S$. By the argument at the beginning of the proof, any $D_k$, $k\in H$, can be written as $p^tX_k+pY_k$ for some $X_k, Y_k \in \mathbb{Z}[\zeta_h][S]$. Therefore, the equation (\ref{form of D}) follows, as $D=\sum_{k\in H} D_kk$.

\medskip

\end{proof}

\bigskip


\subsection{Number Theoretic Results}
We start this subsection with the following definition.

\begin{defi} \label{selfconjugate definition}
Let $p$ be a prime, let $n$ be a positive integer,
and write $n=p^an'$, where $\gcd(p,n')=1$. We call $p$  \textbf{self-conjugate modulo n} if there exists an integer $j$ such that $p^j \equiv -1 \pmod{n'}$. A composite integer $m$ is self-conjugate modulo $n$ if every prime divisor of $m$ has this property.
\end{defi}

\begin{re}\cite[Proposition 2.11]{leu2} \label{selfconjugacy divisibility}
Let $X=\sum_{i=0}^{m-1} a_i\zeta_m^i \in \mathbb{Z}[\zeta_m]$ so that $X\bar{X}=n$. Let $u$ be the largest divisor of $n$ which is self-conjugate modulo $m$. Write $u=w^2k$, where $k=\prod_{i=1}^r p_i$ is the square-free part of $u$. Then 
$$m\equiv 0 \pmod{k}.$$
Furthermore, for  $i=1,\dots, r$, write
$$\Theta_i=\begin{cases} 1-\zeta_4 \ \text{if} \ p_i=2,\\ \sum_{j=1}^{p_i-1} \left(\frac{j}{p_i}\right) \zeta_{p_i}^j \ \text{otherwise}.\end{cases}$$
where $(-)$ is the Legendre symbol. Then
$$X\equiv 0 \pmod{w\prod_{i=1}^r \Theta_i}.$$
\end{re}

\begin{defi} \label{field descent F(m,n)}
Let $m$ and $n$ be positive integers and let $m=\prod_{i=1}^t p_i^{c_i}$ be the prime factorization of $m$. For each prime divisor $q$ of $n$, define
$$\widetilde{m}_q =\begin{cases} \prod_{p_i \neq q} p_i \ \text{if} \ m \ \text{is odd} \ \text{or} \ q=2, \\
								4\prod_{p_i\neq 2,q} p_i \ \text{otherwise}.
			\end{cases}$$
Let $D(n)$ denote the set of all prime divisors of $n$. Let
$$F(m,n)=\prod_{i=1}^t p_i^{b_i}$$
be the minimum multiple of $\prod_{i=1}^t p_i$ such that for every pair $(i,q)$ with $1\leq i\leq t$ and $q\in D(n)$, at least one of the following conditions is satisfied.
\begin{itemize}
\item[(a)] $q=p_i$ and $(p_i,b_i)\neq (2,1)$, or
\item[(b)] $b_i=c_i$, or
\item[(c)] $q\neq p_i$ and $q^{\ord_{\widetilde{m}_q}(q)} \not\equiv 1 \pmod{p_i^{b_i+1}}$.
\end{itemize}
\end{defi}

\begin{re}\cite[Theorem 3.5]{sc3} \label{field descent schmidt}
Let $m$ and $n$ be positive integers. If $X \in \mathbb{Z}[\zeta_m]$ satisfies $X\bar{X}=n$, then there exists an integer $j$ such that $X\zeta_m^j \in \mathbb{Z}[\zeta_{F(m,n)}]$. 
 \end{re}

\medskip

 \subsection{Bounds on Norms of Cycloctomic Integers}
We start this subsection with a bound on the norm of a cyclotomic integer. This bound appeared in \cite[Theorem 3.1]{duc} and will serve as the main ingredient for our results later. We provide the proof here for the convenience of the reader.

\begin{re} \label{bound for cyclotomic integer}
Let $\alpha=\sum_{i=0}^{m-1}c_i \zeta_m^i, \ c_i \in \mathbb{Z},$ be an element of $\mathbb{Z}[\zeta_m]$. Then
\begin{equation} \label{bound for norm of cyclotomic integer}
\left | N_{\mathbb{Q}(\zeta_m)/\mathbb{Q}}(\alpha) \right | \leq \left( \frac{m}{\varphi(m)}\sum_{i=0}^{m-1}c_i^2\right)^{\varphi(m)/2}.
\end{equation}
\end{re}

\begin{proof}
Put $f(x)=\sum_{i=0}^{m-1}c_ix^i \in \mathbb{Z}[x]$, then $\alpha=f(\zeta_m)$ and the conjugates of $\alpha$ in $\mathbb{Q}(\zeta_m)$ are $f(\zeta_m^t)$, where $t$ is coprime to $m$. We have
$$\sum_{t=0}^{m-1} |f(\zeta_m^t)|^2=\sum_{i,j,t=0}^{m-1}c_ic_j\zeta_m^{(i-j)t}=m\sum_{i=0}^{m-1}c_i^2.$$
By the inequality of arithmetic and geometric means, we obtain
\begin{eqnarray*}
|N(f(\zeta_m))| &=& \left|\vcenter{ \hbox{$\displaystyle \prod_{\gcd(t,m)=1} f(\zeta_m^t)$}} \right| \\
 &\leq & \left(\frac{\sum_{(t,m)=1}|f(\zeta_m^t)|^2}{\varphi(m)}\right)^{\varphi(m)/2} \\
& \leq &  \left(\frac{m}{\varphi(m)}\sum_{i=0}^{m-1}c_i^2 \right)^{\varphi(m)/2}.
\end{eqnarray*}
\end{proof}

The next result can be derived from \cite[Theorem 2.3.2]{sc2}. We state the result  below and give a simple proof which makes use of Result \ref{bound for cyclotomic integer}.

\begin{re} \label{weil field descent}
Let $m$ and $n$ be positive integers. Let $X=\sum_{i=0}^{m-1}c_i\zeta_m^i$ be an element of $\mathbb{Z}[\zeta_m]$ such that $X\bar{X}=n$. Let $k$ be a divisor of $m$ which is divisible by $F(m,n)$. Then there exists an integer $j$ such that 
\begin{equation} \label{descent}
X\zeta_m^j \in \mathbb{Z}[\zeta_k].
\end{equation}
Moreover, put $d_i=c_{im/k-j}$, $0\leq i\leq k-1$. Then
 \begin{equation} \label{relation n and k}
n \leq \min\Bigg\{ \left(\sum_{i=0}^{k-1} |d_i|\right)^2, \left(\frac{k}{\varphi(k)}\sum_{i=0}^{k-1}d_i^2\right)\Bigg\}.
\end{equation}
\end{re}

\begin{proof}
Since $k$ is divisible by $F(m,n)$, we obtain $X\zeta_m^j \in \mathbb{Z}[\zeta_k]$, as $X\zeta_m^j \in \mathbb{Z}[\zeta_{F(m,n)}]$ by Result \ref{field descent schmidt}. Note that $k$ is divisible by every prime divisor of $m$, as $F(m,n)$ is. Thus $1,\zeta_m,\dots, \zeta_m^{m/k-1}$ are independent over $\mathbb{Q}(\zeta_k)$. This implies $X\zeta_m^j=\sum_{i=0}^{k-1}d_i\zeta_k^i$, where $d_i=c_{im/k-j}$. Clearly, we have
$$n=|X\bar{X}|\leq \left(\sum_{i=0}^{k-1} |d_i|\right)^2.$$
The other part of (\ref{relation n and k}) follows directly from (\ref{bound for norm of cyclotomic integer}).
\end{proof}

\bigskip


\section{Necessary Conditions}
We will focus only on two types of matrices: $\BH(\mathbb{Z}_n,h)$ matrices and $\BH(G,2p^b)$ matrices, where $p$ is an odd prime and $n,h,b$ are positive integers.

\subsection{Existence of {\boldmath $\BH(\mathbb{Z}_n,h)$} Matrices}
Note that by Result \ref{TaiSchmidt}, there exist $\BH(\mathbb{Z}_n,h)$ matrices whenever $n$ divides $(h,n)^2$ and $n$ and $h$ are not both congruent to $2$ modulo $4$. We conjecture that these conditions are also necessary for the existence of $\BH(\mathbb{Z}_n,h)$ matrices.

\begin{conj} \label{cir but}
Let $n$ and $h$ be positive integers. Then there exists a $\BH(\mathbb{Z}_n,h)$ matrix if and only if
\begin{itemize}
\item[(i)] $\nu_p(h)\geq \lceil \nu_p(n)/2 \rceil$ for every prime divisor $p$ of $n$, and
\item[(ii)] $\nu_2(h) \geq 2$ if $n\equiv 2 \pmod{4}$.
\end{itemize}
\end{conj}

A special case of Conjecture \ref{cir but} is the circulant Hadamard matrix conjecture which was mentioned in the introduction. Unfortunately, we are far from proving Conjecture \ref{cir but}. In support of it, we prove that $n \leq (h,n)^2$ under certain restrictions on $n$ and $h$.

\medskip

\begin{thm} \label{inequality relation between n and h}
Let $n$ and $h$ be positive integers and put $m=\lcm(n,h)$. Suppose that a $\BH(\mathbb{Z}_n,h)$ matrix exists. Furthermore, assume that for any prime divisor $p$ of $n$, we have
\begin{itemize}
\item[(i)] $p$ divides $h$, and
\item[(ii)] $q^{\ord_{\widetilde{m}_q}(q)} \not\equiv 1 \pmod{p^{\nu_p(h)+1}} $ for any prime divisor $q\neq p$ of $n$,
\end{itemize}
where $\widetilde{m}_q$ is defined in Definition \ref{field descent F(m,n)}. Then
\begin{equation} \label{inequality}
n\leq (h,n)^2.
\end{equation}
\end{thm}

\begin{proof}
Let $H$ be a $\BH(\mathbb{Z}_n,h)$ matrix and let its first row be $(\zeta_h^{a_0},\dots,\zeta_h^{a_{n-1}})$, $a_i \in \mathbb{Z}$. Let $g$ be a generator of the cyclic group $\mathbb{Z}_n$ and let $D=\sum_{i=0}^{n-1}\zeta_h^{a_i}g^i$ be an element of the group ring $\mathbb{Z}[\zeta_h][\mathbb{Z}_n]$. We obtain, by Result \ref{equiv},
\begin{equation} \label{difference set equation}
DD^{(-1)}=n.
\end{equation}
Let $\chi$ be a character of $\mathbb{Z}_n$ such that $\chi(g)=\zeta_n$. We have $\chi(D)=\sum_{i=0}^{n-1}\zeta_h^{a_i}\zeta_n^i \in \mathbb{Z}[\zeta_m].$ 
By Result \ref{field descent schmidt}, there exists an integer $j$ such that $\chi(D)\zeta_m^j \in \mathbb{Z}[\zeta_{F(m,n)}]$. Since $\gcd(m/h,m/n)=1$, there exist integers $\alpha$ and $\beta$ such that $\alpha m/h+\beta m/n=j$. Replacing $D$ by $D\zeta_h^{\alpha}g^{\beta}$, if necessary, we can assume that 
$$X=\chi(D\zeta_h^{\alpha}g^{\beta})=\chi(D)\zeta_n^{\alpha n/h+\beta}=\chi(D)\zeta_m^j \in \mathbb{Z}[\zeta_{F(m,n)}].$$ 
Thus $X \in \mathbb{Z}[\zeta_h]$, as $h\equiv 0 \pmod{F(m,n)}$ by the definition of $F(m,n)$ and by the conditions $(i)$ and $(ii)$. Note that with $h$ in the place of $k$, the condition (\ref{descent}) in Result \ref{weil field descent} is satisfied. To apply the inequality (\ref{relation n and k}), we need to find the exponents in
$$X=\sum_{i=0}^{n-1} \zeta_m^{a_i(m/h)+i(m/n)}$$
which are divisible by $m/h$. These are the ones containing $i$ such that $n/(h,n)$ divides $i$. There are $(h,n)$ such exponents, as $0\leq i \leq n-1$. Recall that $X \in \mathbb{Z}[\zeta_h]$. When expressing $X$ in the form $\sum_{i=0}^{h-1}d_i\zeta_h^i$, $d_i \in \mathbb{Z}$, we have $d_i\geq 0$ for all $i$ and $\sum_{i=0}^{h-1} d_i=(h,n)$. We obtain, by (\ref{relation n and k}),
$$n\leq \left(\sum_{i=0}^{h-1} d_i \right)^2=(h,n)^2.$$
\end{proof}

We remark that in the case $n=p^a$ is a prime power, the inequality (\ref{inequality}) implies that $\nu_p(h) \geq \lceil \nu_p(n)/2\rceil$, which is exactly the second statement in Result \ref{TaiSchmidt}. Moreover in this case, the condition $(ii)$ holds automatically and the condition $(i)$ can be obtained using an argument from the field-descent method.

\bigskip

\noindent For the next result, we recall the self-conjugacy concept in Definition \ref{selfconjugate definition}. We concentrate on prime divisors of $n$ which are self-conjugate modulo $m=\lcm(n,h)$.

\begin{thm} \label{divisibility selfconjugate}
Let $n$ and $h$ be positive integers and put $m=\lcm(n,h)$. If a $\BH(\mathbb{Z}_n,h)$ matrix exists, then any prime divisor of $n$ which is self-conjugate modulo $m$ divides $h$.
\end{thm}

\begin{proof}
Suppose that there exists a prime divisor $p$ of $n$ which is self-conjugate modulo $m$ and does not divide $h$. Letting $g$ be a generator of $\mathbb{Z}_n$ and defining $D=\sum_{i=0}^{n-1} \zeta_h^{a_i}g^i$ as in the proof of Theorem \ref{inequality relation between n and h}, we have $DD^{(-1)}=n.$
For any character $\chi$ of $\mathbb{Z}_n$, we have
 \begin{equation} \label{weil equation D}
 |\chi(D)|^2=n, \ \chi(D) \in \mathbb{Z}[\zeta_m].
 \end{equation}
Applying $\chi_0$ to (\ref{weil equation D}) and putting $Y=\chi_0(D) \in \mathbb{Z}[\zeta_h]$, we obtain $|Y|^2=n$. Note that $p$ is self-conjugate modulo $m$ and $h$ divides $m$, so $p$ is also self-conjugate modulo $h$. By Result \ref{selfconjugacy divisibility} and by the condition $p \nmid h$, we obtain $\nu_p(n)=2t$ for some $t\in \mathbb{Z}^+$. We also have $\nu_p(m)=2t$, as $m=\lcm(n,h)$ and $h\not\equiv 0\pmod{p}$. Moreover, since $|\chi(D)|^2=n$ and $\chi(D)\in \mathbb{Z}[\zeta_m]$, Result \ref{selfconjugacy divisibility} implies $\chi(D)\equiv 0\pmod{p^t}$ for any character $\chi$ of $\mathbb{Z}_n$. By Result \ref{ma lemma}, we obtain
\begin{equation} \label{form of D 2}
D=p^tX+PY,
\end{equation}
where $P$ is the cyclic subgroup of $\mathbb{Z}_n$ of order $p$, and $X, Y \in \mathbb{Z}[\zeta_h][\mathbb{Z}_n]$.
As $kP=P$ for any $k\in P$, we can assume that no two elements in the support of $Y$ are contained in the same coset of $P$ in $\mathbb{Z}_n$.
Comparing the coefficients on a fixed coset of $P$, the equation (\ref{form of D 2}) implies
$$\zeta_h^{a_i}\equiv \zeta_h^{a_{i+n/p}} \equiv \cdots \equiv \zeta_h^{a_{i+(p-1)n/p}} \pmod{p^t} \ \text{for any} \ 0\leq i \leq n/p-1.$$
Hence $\zeta_h^{a_j}\equiv \zeta_h^{a_{j+n/p}} \pmod{p^t}$ for all $j$. Suppose that $\zeta_h^{a_j}=\zeta_h^{a_{j+n/p}}$ for all $j$. We have $D=PZ$, where $Z=\sum_{i=0}^{n/p-1} \zeta_h^{a_i}g^i$. Let $\tau$ be a character of $\mathbb{Z}_n$ such that $\tau(g)=\zeta_n$. We have 
$$\tau(P)=\sum_{i=0}^{p-1}\tau(g^{in/p})=\sum_{i=0}^{p-1}\zeta_p^i=0,$$ 
so $\tau(D)=0$, contradicting  $(\ref{weil equation D})$. Therefore, there exists $j$ such that $\zeta_h^{a_j} \neq \zeta_h^{a_{j+n/p}}$. 

The condition $\zeta_h^{a_j}\equiv \zeta_h^{a_{j+n/p}} \pmod{p^t}$ implies that $t=1$ and $p=2$, as $|\zeta_h^{a_j}-\zeta_h^{a_{j+n/p}}| \leq 2$. The congruence $\zeta_h^{a_j}\equiv \zeta_h^{a_{j+n/2}} \pmod{2}$ holds only when $2 \mid h$ and $a_{j+n/2}=a_j+h/2$, contradicting with the assumption that $p$ does not divide $h$.
\end{proof}

\begin{cor} \label{necessary condition in the circulant case with (n,h)=1}
If $n$ and $h$ are coprime positive integers such that a $\BH(\mathbb{Z}_n,h)$ matrix exists, then no prime divisor of $n$ is self-conjugate modulo $nh$.
\end{cor}

\medskip

\noindent In preparation for the next result of this section, we start with the following definition of delta function.

\begin{defi} \label{delta}
Let $s$ and $t$ be any two complex numbers. We define the function $\delta_{ts}$ as follows.
$$\delta_{ts}=\begin{cases} 0 \ \text{if} \ t\neq s, \\ 1 \ \text{if} \ t=s. \end{cases}$$
\end{defi}

\noindent The following result is taken from \cite[Lemma $2.5$]{sc3}.

\begin{re} \label{integral basis}
Let $m\in \mathbb{Z}^+$ and let $k$ be a divisor of $m$. Let $t$ and $s$ be the numbers of prime divisors of $m$ and $k$, respectively. Write $X=\sum_{i=0}^{m-1}a_i\zeta_m^i$, where $a_i \in \mathbb{Z}$ for all $i$. Then there exists an integral basis $\mathcal{B}_{m,k}$ of $\mathbb{Q}(\zeta_m)$ over $\mathbb{Q}(\zeta_k)$ which contains $\varphi(m/k)$ roots of unity. Furthermore, if $0 \leq a_i \leq C$ for all $i$, then we can express $X$ as
\begin{equation} \label{integral expression of X over subfield}
X=\sum_{x \in \mathcal{B}_{m,k}} x\left(\sum_{j=0}^{k-1}c_{xj}\zeta_k^j\right),  \ |c_{xj}| \leq 2^{t-s-1+\delta_{ts}}C \ \text{for all} \ x,j.
\end{equation}
\end{re}

\medskip

\begin{thm} \label{relation selfconjugate part}
Let $n$ and $h$ be positive integers and put $m=\lcm(n,h)$. Let $u$ be the largest divisor of $n$ which is self-conjugate modulo $m$. Write $u=w^2k$, where $k$ is the square-free part of $u$. Let $t$ and $r$ be the numbers of prime divisors of $h$ and $k$, respectively. Suppose that a $\BH(\mathbb{Z}_n,h)$ matrix exists. We obtain the following.
\begin{itemize}
\item[(a)] If $k\not\equiv 0\pmod{2}$, then 
\begin{equation} \label{inequality1}
w \leq  2^{t-r-1+\delta_{rt}}(h,u)\sqrt{\frac{k}{\varphi(k)}}.
\end{equation}
\item[(b)] If $k\equiv 0 \pmod{2}$, then
\begin{equation} \label{inequality2}
w \leq  2^{t-r-1}(h,u)\sqrt{\frac{2k}{\varphi(k)}}.
\end{equation}
\end{itemize}
 \end{thm}

 \begin{proof}
Similar to the proof of Theorem \ref{inequality relation between n and h}, we have $DD^{(-1)}=n$, where $D=\sum_{i=0}^{n-1}\zeta_h^{a_i}g^i \in \mathbb{Z}[\zeta_h][\mathbb{Z}_n]$ and $g$ is a generator of $\mathbb{Z}_n$. Hence for any character $\chi$ of $\mathbb{Z}_n$, we have
\begin{equation} \label{weil equation for D}
 |\chi(D)|^2=n, \ \chi(D)\in \mathbb{Z}[\zeta_m].
\end{equation}
Put $v=n/u$. Note that $(u,v)=1$, as $u$ is the largest divisor of $n$ which is self-conjugate modulo $m$. Let $\tau$ be a character of $\mathbb{Z}_n$ such that $\tau(g^v)=\zeta_u$. The proof of the theorem is divided into several claims.

\medskip

\noindent \underline{Claim 1}. Replacing $D$ by $Dy$, $y \in G$, if necessary, we can assume that
\begin{equation} \label{nonzero}
\tau(D\cap \mathbb{Z}_u)\neq 0.
\end{equation}
Let $R=\{1,g,\dots,g^{v-1}\}$ be the complete set of coset representatives of $\mathbb{Z}_u$ in $\mathbb{Z}_n$. If $\tau(D\cap x\mathbb{Z}_u)=0$ for all $x\in R$, then $\tau(D)=\sum_{x\in R}\tau(D\cap x\mathbb{Z}_u)=0$, contradicting (\ref{weil equation for D}). Thus, there exists $x\in R$ such that $\tau(D\cap x\mathbb{Z}_u)\neq 0$. Note that
$\tau(Dx^{-1}\cap\mathbb{Z}_u)=\tau(x^{-1})\tau(D\cap x\mathbb{Z}_u) \neq 0$. So, replacing $D$ by $Dx^{-1}$ if necessary, we can assume that $\tau(D \cap \mathbb{Z}_u) \neq 0,$ proving (\ref{nonzero}).

\medskip

\noindent \underline{Claim 2}. Write $D\cap \mathbb{Z}_u=\sum_{i=0}^{u-1} \zeta_h^{a_{vi}}g^{vi}=\sum_{i=0}^{u-1}\zeta_h^{b_i}g^{vi}$ and write $k=\prod_{i=1}^r p_i$. Define $\Theta_i$ as in Result \ref{selfconjugacy divisibility}. We have
\begin{equation} \label{divisible}
X=\sum_{i=0}^{u-1}\zeta_h^{b_i}\zeta_u^i\equiv 0\pmod{w\prod_{i=1}^r \Theta_i}.
\end{equation}

\noindent By Result \ref{selfconjugacy divisibility}, we have
 $$\chi(D) \equiv 0 \pmod{w\prod_{i=1}^r \Theta_i}$$ 
 for any character $\chi$ of $\mathbb{Z}_n$. On the other hand, by Result \ref{orthogonal sum}, we have
\begin{equation} \label{orthogonal}
\sum_{\chi \in \mathbb{Z}_u^\perp} \chi\tau(D)=|\mathbb{Z}_u^\perp|\tau(D\cap \mathbb{Z}_u)=v\left(\sum_{i=0}^{u-1}\zeta_h^{b_i}\zeta_u^i\right),
\end{equation}
Since $(u,v)=1$ and each term on the left side of (\ref{orthogonal}) is divisible by $w\prod_{i=1}^r \Theta_i $, we obtain the congruence (\ref{divisible}).

\medskip

\noindent \underline{Claim 3.} Put $l=\lcm(h,u)$, $d=\gcd(h,u)$ and write $X=\sum_{i=0}^{u-1}\zeta_h^{b_i}\zeta_u^i$ in the form
$X=\sum_{i=0}^{l-1}c_i\zeta_l^i, \ c_i \in \mathbb{Z}^+$. Then 
\begin{equation} \label{claim bound for ci}
0\leq c_i \leq d \ \ \text{for any} \ \ i=0,\dots,l-1.
\end{equation}
For a fixed $0\leq i \leq l-1$, let $a$ and $b$ be fixed integers such that
$(u/d) a+(h/d) b=i$. All the solutions $j \pmod{h}$ and $f \pmod{u}$ to $\zeta_l^i=\zeta_h^j\zeta_u^f$ are
$$j=a+l\frac{h}{d}, \ f=b-l\frac{u}{d}, \ 0\leq l \leq d-1.$$
Thus, there are at most $d$ solutions to $\zeta_l^i=\zeta_h^{b_z}\zeta_u^z$, $0\leq z\leq u-1$. The claim is proved.

\medskip

\noindent \underline{Claim 4}. If $k \not\equiv 0 \pmod{2}$, then
$$w \leq 2^{t-r-1+\delta_{tr}}(h,u)\sqrt{\frac{k}{\varphi(k)}}.$$
First, note that $X=\tau(D\cap \mathbb{Z}_u)\neq 0$ by Claim 1. By Theorem \ref{divisibility selfconjugate}, any prime divisor of $u$ divides $h$, so the numbers of prime divisors of $l=\lcm(h,u)$ and $h$ are the same, both are equal to $t$. By Result \ref{integral basis} and (\ref{claim bound for ci}), we can express $X$ as a linear combination of the basis elements in $\mathcal{B}_{l,k}\subset \mathbb{Q}(\zeta_l)$ over $\mathbb{Q}(\zeta_k)$ as follows
$$X=\sum_{x\in \mathcal{B}_{l,k}} x\left( \sum_{j=0}^{k-1} c_{xj} \zeta_k^j \right),$$
where
\begin{equation} \label{inequality for c_xj}
|c_{xj}| \leq 2^{t-r-1+\delta_{tr}}d \ \forall \ x\in \mathcal{B}_{l,k}, \ 0\leq j \leq k-1.
\end{equation}
Note that $X \equiv 0 \pmod{w\prod_{i=1}^r \Theta_i}$ by (\ref{divisible}). Moreover note that $w\prod_{i=1}^r\Theta_i \in \mathbb{Q}(\zeta_k)$, as $\Theta_i\in \mathbb{Q}(\zeta_{p_i})$ for all $i$ by the definition of $\Theta_i$. Hence $\sum_{j=0}^{k-1} c_{xj}\zeta_k^j \equiv 0\pmod{w\prod_{i=1}^r \Theta_i}$ for all $x \in \mathcal{B}_{l,k}$. As $X \neq 0$, there exists $x \in \mathcal{B}_{l,k}$ so that 
\begin{equation} \label{divisible resulting condition}
Y=\sum_{j=0}^{k-1} c_{xj}\zeta_k^j \neq 0 \ \text{and} \ Y \equiv 0 \pmod{w\prod_{i=1}^r \Theta_i}.
\end{equation}
Note that each $\Theta_i$ has absolute value $\sqrt{p_i}$. Using (\ref{inequality for c_xj}), (\ref{divisible resulting condition}) and the inequality (\ref{bound for norm of cyclotomic integer}), we obtain
\begin{eqnarray*}
(w^2k)^{\varphi(k)/2}=\Big| N_{\mathbb{Q}(\zeta_k)/\mathbb{Q}}(w\prod_{i=1}^r\Theta_i)\Big| &\leq & \Big| N_{\mathbb{Q}(\zeta_k)/\mathbb{Q}}(Y) \Big| \leq \left(\frac{k}{\varphi(k)}\sum_{j=0}^{k-1} c_{xj}^2\right) ^{\varphi(k)/2}\\
&\leq & \left( 4^{t-r-1+\delta_{rt}}d^2\frac{k^2}{\varphi(k)}\right)^{\varphi(k)/2},
\end{eqnarray*}
proving Claim 4.

\medskip

\noindent \underline{Claim 5}. If $k\equiv 0 \pmod{2}$, then
$$w \leq 2^{t-r-1}(h,u)\sqrt{\frac{2k}{\varphi(k)}}.$$
We assume that $p_r=2$. Similar to the proof of Claim 4, we express $X$ as a linear combination of the basis elements in $\mathcal{B}_{l,k/2}\subset \mathbb{Q}(\zeta_l)$ over $\mathbb{Q}(\zeta_{k/2})$
$$X=\sum_{x\in \mathcal{B}_{l,k/2}} x\left(\sum_{j=0}^{k/2-1} d_{xj}\zeta_{k/2}^j\right), \ |d_{xj}| \leq 2^{t-r}d \ \forall \ x,j.$$
Note that $w\prod_{i=1}^{r-1}\Theta_i\in \mathbb{Q}(\zeta_{k/2})$ and divides $X$, so it divides $\sum_{j=0}^{k/2-1}d_{xj}\zeta_{k/2}^j$ for any $x\in \mathcal{B}_{l,k/2}$. As $X\neq 0$, there exists $x\in \mathcal{B}_{l,k/2}$ such that
$$Z=\sum_{j=0}^{k/2-1}d_{xj}\zeta_{k/2}^j\neq 0 \ \text{and} \ Z \equiv 0\pmod{w\prod_{i=1}^{r-1}\Theta_i}.$$
Using (\ref{inequality for c_xj}), (\ref{divisible resulting condition}) and the inequality (\ref{bound for norm of cyclotomic integer}), we obtain
\begin{eqnarray*}
(w^2k/2)^{\varphi(k/2)/2} &=& \Big| N_{\mathbb{Q}(\zeta_{k/2})/\mathbb{Q}}(Z) \Big| \leq \left(\frac{k/2}{\varphi(k/2)}\sum_{j=0}^{k/2-1} d_{xj}^2\right) ^{\varphi(k/2)/2}\\
&\leq & \left( 4^{t-r}d^2\frac{k^2}{4\varphi(k)}\right)^{\varphi(k/2)/2},
\end{eqnarray*}
proving Claim 5.

 \end{proof}

\medskip

\begin{cor} \label{h is a prime power}
Let $n$ and $b$ be positive integers and let $p$ be a prime. Suppose that a $\BH(\mathbb{Z}_n,p^b)$ matrix exists. Then $n=p^cm$ for some positive integers $c$ and $m$ in which $m\not\equiv 0\pmod{p}$. Moreover, assume that $p$ is selfconjugate modulo $m$. Then
\begin{equation} \label{exponent_p}
b\geq \lfloor c/2 \rfloor. 
\end{equation}
\end{cor}
\begin{proof}
The claim that $n=p^cm$ follows directly from Result \ref{Lam-Leung}.
 We apply the inequality (\ref{inequality}) to prove the second claim. In this case, we have $u=p^c$, $w=p^{\lfloor c/2 \rfloor}$, $t=1$, $k \in \{1,p\}$ and $r \in \{0,1\}$. 
 
 If $k\equiv 0 \pmod{2}$, then $k=p=2$, $r=1$ and the inequality (\ref{inequality2}) implies
$$p^{\lfloor c/2 \rfloor} \leq (p^b,p^c)\leq p^b,$$
proving (\ref{exponent_p}).

If $k\not\equiv 0\pmod{2}$, then $p\geq 3$ and the inequality (\ref{inequality1}) implies
$$p^{\lfloor c/2 \rfloor} \leq 2^{-r+\delta_{r1}}(p^b,p^c)\sqrt{\frac{k}{\varphi(k)}}\leq (p^b,p^c)\sqrt{\frac{p}{p-1}}<p^{b+1},$$
proving (\ref{exponent_p}).
\end{proof}

\medskip

\begin{rmk} \label{first three conditions}
We consider the impact of Theorem \ref{inequality relation between n and h}, Theorem \ref{divisibility selfconjugate} and Theorem \ref{relation selfconjugate part} on the existence of $\BH(\mathbb{Z}_n,h)$ matrices with $2\leq n,h \leq 100$.
 Theorem \ref{inequality relation between n and h} confirms the non-existence of $178$ cases. Theorem \ref{divisibility selfconjugate} confirms the nonexistence of $2946$ cases. Lastly, Theorem \ref{relation selfconjugate part} confirms the nonexistence of $723$ cases.
\end{rmk}

\bigskip

\subsection{Existence of {\boldmath $\BH(G,2p^b)$} Matrices}
Let $G$ be an abelian group, let $p$ be an odd prime and let $b$ be a positive integer. The main result of this section relies on the following result by Leung and Schmidt, see \cite[Theorem 22 and Theorem 23]{leu1} .

\begin{re} \label{other conditions}
Let $p$ be an odd prime and let $a$ be a positive integer. Let $m$ be a nonsquare integer and let $q_1,\dots,q_s$ be all distinct prime divisors of $m$. Put $f=\gcd(\ord_p(q_1),\dots, \ord_p(q_s))$. Suppose that $X\in \mathbb{Z}[\zeta_{p^a}]$ satisfies $|X|^2=m$. Then the following hold.
\begin{itemize}
\item[(i)] $f$ is an odd integer.
\item[(ii)] Either $f\leq m$ or $p\leq (f^2-m)/(f-m)$.
\item[(iii)] $p\leq m^2+m+1$.
\end{itemize}
\end{re}

\begin{thm} \label{prime power h other condition}
Let $b$ be a positive integer, let $p$ be an odd prime and let $G$ be an abelian group. Write $|G|=p^cm$, where $c\geq 0$ and $p$ does not divide $m$. Suppose that $m$ is not a square. Let $q_1,\dots, q_s$ be all distinct prime divisors of $m$. Put $f=\gcd(\ord_p(q_1),\dots,ord_p(q_s))$. If either a $\BH(G,p^b)$ matrix or a $\BH(G,2p^b)$ matrix exists, then the following hold.
\begin{itemize}
\item[(i)] $f$ is odd.
\item[(ii)] Either $f \leq m$ or $p \leq (f^2-m)/(f-m)$.
\item[(iii)] $p \leq m^2+m+1$.
\end{itemize}
\end{thm}
\begin{proof}
As $\BH(G,p^b) \subset \BH(G,2p^b)$, it suffices to assume that a $\BH(G,2p^b)$ matrix $H$ exists. Let the first row of $H$ be $(\zeta_{2p^b}^{a_g})_{g\in G}$, $a_g \in \mathbb{Z}$ for all $g\in G$. By Result \ref{equiv}, we have
$$DD^{(-1)}=p^cm, \ \ D=\sum_{g\in G}\zeta_{2p^b}^{a_g}g \in \mathbb{Z}[\zeta_{2p^b}][G].$$
Put $X=\chi_0(D)=\sum_{g\in G} \zeta_{2p^b}^{a_g}$. We obtain
$$X\bar{X}=p^cm, \ X\in \mathbb{Z}[\zeta_{2p^b}]=\mathbb{Z}[\zeta_{p^b}].$$
In the ring $\mathbb{Z}[\zeta_{p^b}]$, we have $p\mathbb{Z}[\zeta_{p^b}]=(1-\zeta_{p^b})^{\varphi(p^b)}$ and $(1-\zeta_{p^b})$ is the only prime ideal above $p$. Thus $(1-\zeta_{p^b})^{\varphi(p^b)c/2}$ divides $(X)$ as ideals of $\mathbb{Z}[\zeta_{p^b}]$. Putting
$Y=X(1-\zeta_{p^b})^{-\varphi(p^b)c/2} \in \mathbb{Z}[\zeta_{p^b}]$, we obtain
$$Y\bar{Y}=m.$$
The conclusion follows directly from Result \ref{other conditions}.
\end{proof}

\medskip

\begin{cor}
If $G$ is an abelian group with $|G|=2p^c$ for some non-negative integer $c$ and odd prime $p$, then there is no $\BH(G,2p^b)$ matrix.
\end{cor}
\begin{proof}
Using Theorem \ref{prime power h other condition} part $(iii)$ for $m=2$, we obtain $p \leq 7$. The case $p=3$ or $p=5$ cannot satisfy the condition $f=\ord_p(2)$ is odd. Therefore, the only possible value for $p$ is $p=7$. The case $p=7$ was also ruled out by Leung and Schmidt \cite{leu1}.
\end{proof}

\noindent We note that Theorem \ref{prime power h other condition} confirms the nonexistence of $2361$ $\BH(\mathbb{Z}_n,h)$ matrices in which $2\leq n, h\leq 100$.

\medskip

\begin{rmk} \label{ending rmk}
Using Results \ref{nonexistence}, \ref{sylvester conditions}, \ref{Lam-Leung} and \ref{Brock}, we have $5108$ open cases for the existence of a $\BH(\mathbb{Z}_n,h)$ matrix in which $1 \leq n,h \leq 100$. Result \ref{TaiSchmidt} settles the  existence problem  
of $1798$ of these cases. There are $3310$ cases left. 
Theorem \ref{inequality relation between n and h}, Theorem \ref{divisibility selfconjugate}, Theorem \ref{relation selfconjugate part} and Theorem \ref{prime power h other condition} confirm the nonexistence of $623$ cases in the $3310$ open cases above. 

In summary, there are $2687$ open cases for the existence of $\BH(\mathbb{Z}_n,h)$ matrices in which $1\leq n,h \leq 100$. This list of open cases is provided in \cite{duc2}.
\end{rmk}

\bigskip


\section{Application to Perfect Polyphase Arrays }

We recall from the introduction that a perfect $h$-phase sequence of length $n$ is equivalent to a $\BH(\mathbb{Z}_n,h)$ matrix.
Theorem \ref{inequality relation between n and h}, Theorem \ref{divisibility selfconjugate}, Theorem \ref{relation selfconjugate part} and Theorem \ref{prime power h other condition} provide various necessary conditions for the existence of such sequences. In Lemma \ref{equivalence} below, we prove that a perfect $h$-phase array of size $n_1\times\cdots\times n_k$ is equivalent to a $\BH(\mathbb{Z}_{n_1}\times \dots \times \mathbb{Z}_{n_k},h)$ matrix.

\medskip

\noindent The following table gives parameters of known perfect arrays, see \cite{blake}.

\begin{center}
\begin{tabular}{| c | c | c | }
\hline
Author & Array size & Alphabet size (value of $h$) \\
\hline
Jedwab, Mitchell & $2^n\times 2^n$ & 2 \\
\hline
Kopilovich & \begin{tabular}{@{}c@{}} $(3\times 2^{n+1})\times (3 \times 2^{n+1})$ \\ $(3\times 2^n)\times (3 \times 2^{n+2})$
\end{tabular}& $2$ \\
\hline
Wild & $2^n\times 2^{n+2}$ & $2$  \\
\hline
Blake, Hall, Tirkel & $n\times n$ & \begin{tabular}{@{}c@{}} $n$ if $n$ is odd,\\ $2n$ if $n$ is even \end{tabular}\\
\hline
Blake, Hall, Tirkel & $n\times n^{2k+1}$ & \begin{tabular}{@{}c@{}} $n^{k+1}$ if $n$ is odd,\\ $\lcm(2n,n^{k+1})$ if $n$ is even \end{tabular}\\
\hline
Blake, Hall, Tirkel & $n\times n^2$ & \begin{tabular}{@{}c@{}} $n$ if $n$ is odd,\\ $2n$ if $n$ is even \end{tabular}\\
\hline
Blake, Hall, Tirkel & $n^{2k+1}\times n^{2k+1}$ & $n^{k+1}$\\
\hline
Blake, Hall, Tirkel & $n^2\times n^2$ & $n$\\
\hline
\end{tabular}

\end{center}

\medskip

\noindent The following result is a direct corollary of Result \ref{TaiSchmidt}, see also \cite{mow}.

\begin{re} \label{circulant butson}
Let $n$ and $h$ be positive integers. Then a $\BH(\mathbb{Z}_n,h)$ matrix exists whenever $n$ and $h$ satisfy the following condition:
\begin{equation*}
n \mid (h,n)^2 \ \text{and} \ (\nu_2(n),\nu_2(h)) \neq (1,1). 
\end{equation*}
\end{re}

\noindent To prepare for the main result of this section, we need the following lemmas.

\begin{lem} \label{product of group-invariant matrices}
Let $h_1$ and $h_2$ be positive integers and put $h=\lcm(h_1,h_2)$. Let $G_1$ and $G_2$ be finite abelian groups. Assume that $H_1$ is a $\BH(G_1,h_1)$ matrix and $H_2$ is a $\BH(G_2,h_2)$ matrix. Then the Kronecker product $H=H_1\otimes H_2 $ is a $\BH(G_1\times G_2,h)$ matrix.
\end{lem}
\begin{proof}
Assume that $H_1=(h_{x_1,y_1})_{x_1,y_1 \in G_1}$ and $H_2=(k_{x_2,y_2})_{x_2,y_2 \in G_2}$. The elements of the matrix $H=H_1\otimes H_2$ are indexed by $G_1 \times G_2$ in which the $(x_1,x_2)\times (y_1,y_2)$ element is $H_{(x_1,x_2),(y_1,y_2)}=h_{x_1,y_1}k_{x_2,y_2}$. Let $(a_1,a_2)$ be any element in $G_1\times G_2$. We have
\begin{eqnarray*}
H_{(a_1+x_1,a_2+x_2),(a_1+y_1,a_2+y_2)}&=&h_{a_1+x_1,a_1+y_1}k_{a_2+x_2,a_2+y_2} \\
&=& h_{x_1,y_1}k_{x_2,y_2}\\
&=&H_{(x_1,x_2),(y_1,y_2)},
\end{eqnarray*}
proving that $H$ is $G_1\times G_2$-invariant. 

On the other hand, note that each entry of $H$ is a product of a $(h_1)$th root of unity and a $(h_2)$th root of unity, so it is a $h$th root of unity. It remains to verify that $HH^*=|H|I_{|H|}$, or equivalently, any two distinct rows $(a_1,a_2)$ and $(b_1,b_2)$ of $H$ are orthogonal. The inner product of these two rows is
\begin{eqnarray*}
\sum_{x_1\in G_1, x_2\in G_2}H_{(a_1,a_2),(x_1,x_2)}\overline{H}_{(b_1,b_2),(x_1,x_2)}&=&\sum_{x_1\in G_1, x_2\in G_2} h_{a_1,x_1}\overline{h}_{b_1,x_1}k_{a_2,x_2}\overline{k}_{b_2,x_2} \\
&=& \big( \sum_{x_1\in G_1} h_{a_1,x_1}\overline{h}_{b_1,x_1}\big) \big(\sum_{x_2\in G_2}k_{a_2,x_2}\overline{k}_{b_2,x_2}\big) \\
&=& \quad 0,
\end{eqnarray*}
where the last equality holds as $(a_1,b_1)\neq (a_2,b_2)$ and $H_1$ is a $\BH(G_1,h_1)$ matrix and $H_2$ is a $\BH(G_2,h_2)$ matrix.
\end{proof}

\begin{lem} \label{equivalence}
Let $k,h, n_1,\ldots ,n_k$ be positive integers. Then a $\BH(\mathbb{Z}_{n_1}\times \cdots \times \mathbb{Z}_{n_k},h)$ matrix exists if and only if a perfect $h$-phase array of size $n_1\times \cdots \times n_k$ exists.
\end{lem}

\begin{proof}
Put $G=\mathbb{Z}_{n_1}\times \cdots\times \mathbb{Z}_{n_k}$ and suppose that a $\BH(G,h)$ matrix exists. For each $1\leq i \leq k$, let $g_i$ be a generator of $\mathbb{Z}_{n_i}$. By Result \ref{equiv}, there exists $D \in \mathbb{Z}[\zeta_h][G]$ such that $DD^{(-1)}=n_1\cdots n_k$. Define 
$$I=\{(i_1,\dots,i_k): \ 0\leq i_j\leq n_j-1 \ \text{for} \ j=1,\dots,k\}.$$
Write
$$D=\sum_{(i_1,\dots,i_k)\in I}a_{i_1,\dots,i_k}g_1^{i_1}\cdots g_k^{i_k},$$
where each $a_{i_1,\dots,i_k}$ is a complex $h$th root of unity. Note that
$$DD^{(-1)}=\sum_{(s_1,\dots,s_k)\in I} \left(\sum_{(i_1,\dots,i_k)\in I}a_{i_1+s_1,\dots,i_k+s_k}\overline{a}_{i_1,\dots,i_k}\right)g_1^{s_1}\cdots g_k^{s_k},$$
As $DD^{(-1)}=n_1\cdots n_k$, we obtain
\begin{equation} \label{relation}
\sum_{(i_1,\dots,i_k)\in I}a_{i_1+s_1,\dots,i_k+s_k}\overline{a}_{i_1,\dots,i_k}=\begin{cases} n_1\cdots n_k \ \ \text{if} \ s_1=\dots=s_k=0, \\ 0 \ \ \text{otherwise}. \end{cases}
\end{equation}
Define the array $A$ of size $n_1\times\cdots\times n_k$ by
$$A=(a_{i_1,\dots,i_k}), \ 0\leq i_j \leq n_j-1 \ \text{for all} \ j.$$
The equation (\ref{relation}) implies that $A$ is perfect. Conversely, it is straightforward to verify that the existence of a perfect array $A$ implies the existence of a group ring element $D\in \mathbb{Z}[\zeta_h][G]$ whose coefficients are complex $h$th roots of unity such that $DD^{(-1)}=n$, which implies a $\BH(G,h)$ matrix by Result \ref{equiv}.
\end{proof}

\medskip

\begin{thm} \label{newperfectarray}
Suppose that $k,h,n_1,\ldots,n_k$ are positive integers such that
\begin{equation}
n_i \mid (h,n_i)^2 \ \text{and} \ (\nu_2(n_i),\nu_2(h)) \neq (1,1)  \ \text{for any} \ 1\leq i \leq k. \tag{$\star$}
\end{equation}
Then a perfect $h$-phase array of size $n_1\times \cdots \times n_k$ exists.
\end{thm}

\begin{proof}
By Result \ref{circulant butson}, the condition $(\star)$ imply that a $\BH(\mathbb{Z}_{n_i},h)$ matrix exist for any $i=1,\dots,k$. By Lemma \ref{product of group-invariant matrices}, a $\BH(\mathbb{Z}_{n_1}\times \cdots \times \mathbb{Z}_{n_k}, h)$ matrix exists. The desired array is constructed as in Lemma \ref{equivalence}.
\end{proof}

\bigskip

\noindent \textbf{Acknowledgement.} 
The author would like to thank Bernhard Schmidt for his tremendous help and guidance throughout the project.

\end{document}